\documentclass[12pt]{article}
\usepackage{amsmath,amssymb,amsthm,dsfont,tikz,subfigure}
\usetikzlibrary{trees,snakes}

\newenvironment{claimproof}[1][Proof of Claim]
{ \begin{proof}[#1]}
{\end{proof}  }
\newcommand{\p}[1]{\mathbb{P}\left[{#1}\right]}

\newcommand{\e}[1]{\mathbb{E}\left[{#1}\right]}
\newcommand{\node}[1]{{\mathbf{n}}^{#1}}
\newcommand{\dist}{\operatorname{dist}}
\newcommand{\diam}{\operatorname{diam}}
\newcommand{\ah}{\operatorname{ah}}

\newcommand{\cond} {{\,\left\vert\vphantom{\frac{1}{1}}\right.\,}}

\numberwithin{equation}{section}

\theoremstyle{plain}
\newtheorem{theorem}{Theorem}[section]
\newtheorem{lemma}[theorem]{Lemma}

\newtheorem{corollary}[theorem]{Corollary}

\theoremstyle{definition}

\theoremstyle{remark}
\newtheorem*{remark}{Remark}
\newtheorem*{claim}{Claim}

\title{On the Longest Paths and the Diameter \\ in Random Apollonian Networks}
\author{
{Ehsan Ebrahimzadeh\thanks{Department of Electrical and Computer Engineering, University of Waterloo, Waterloo ON, Canada. Email address: \texttt{eebrahim@uwaterloo.ca}.}}
\and Linda Farczadi\thanks{Department of Combinatorics and Optimization, University of Waterloo, Waterloo ON, Canada. Email address: \texttt{lindafarczadi@gmail.com}.}
\and Pu Gao\thanks{Department of Computer Science, University of Toronto, Toronto ON, Canada. Email address: \texttt{jane.p.gao@gmail.com}.}
\and Abbas Mehrabian\thanks{Department of Combinatorics and Optimization, University of Waterloo, Waterloo ON, Canada. Email address: \texttt{amehrabi@uwaterloo.ca}.}
\and Cristiane M.~Sato\thanks{Department of Combinatorics and Optimization, University of Waterloo, Waterloo ON, Canada. Email address: \texttt{cmsato@gmail.com}.}
\and Nick Wormald\thanks{Department of Combinatorics and Optimization, University of Waterloo, Waterloo ON, Canada, and School of Mathematical Sciences, Monash University, Clayton, VIC, Australia. This author acknowledges  support from the Canada Research Chairs program and NSERC, and the ARC Australian Laureate Fellowship scheme. Email address: \texttt{nick.wormald@monash.edu}.}
\and Jonathan Zung\thanks{Department of Mathematics, University of Toronto, Toronto ON, Canada. Email address: \texttt{jonathanzung@gmail.com}.}
}

\begin{document}
\date{}
\maketitle

\begin{abstract}
 We consider the following iterative construction of a random planar triangulation.  Start with a triangle embedded in the plane.
In each step, choose a bounded face uniformly at random,
add a vertex inside that face and join it to the vertices of the face.
After $n-3$ steps, we obtain a random triangulated plane graph with $n$ vertices,
which is called a {Random Apollonian Network (RAN)}.
We show that asymptotically almost surely (a.a.s.) every path in a RAN has length $o(n)$,
refuting a conjecture of Frieze and Tsourakakis.
We also show that a RAN always has a path of length $(2n-5)^{\log 2 / \log 3}$,
and that the expected length of its longest path is $\Omega\left(n^{0.88}\right)$.
Finally, we prove that a.a.s.\ the diameter of a RAN is asymptotic to $c \log n$,
where $c\approx 1.668$ is the solution of an explicit equation.
\end{abstract}

\section{Introduction}

Due to the increase of interest in social networks, the Web graph, biological networks etc.,
in recent years a large amount of research has focused on modelling real world networks
(see, e.g.,  Bonato~\cite {web_survey} or Chung and Lu~\cite{complexgraphs}).
Despite the outstanding amount of work on models generating graphs with power law degree sequences,
a considerably smaller amount of work has focused on generative models for planar graphs.
In this paper we study a popular random graph model for generating planar graphs with power law properties, which is defined as follows.
Start with a triangle embedded in the plane.
In each step, choose a bounded face uniformly at random,
add a vertex inside that face and join it to the vertices on the face.
We call this operation \emph{subdividing} the face.
In this paper, we use the term ``face'' to refer to a ``bounded face,''
unless specified otherwise.
After $n-3$ steps, we have a (random) triangulated plane graph with $n$ vertices and $2n-5$ faces.
This is called a \emph{Random Apollonian Network (RAN)} and we study its asymptotic properties,
as its number of vertices goes to infinity.
The number of edges equals $3n-6$, and hence a RAN is a maximal plane graph.

The term ``{apollonian network}''  refers to a deterministic version of this process,
formed by subdividing all triangles the same number of times,
which was first studied in~\cite{ANs_1,ANs_2}.
Andrade~et~al.~\cite{ANs_1} studied power laws in the degree sequences of these networks.
Random apollonian networks were defined in Zhou~et~al.~\cite{define_RANs} (see Zhang et al.~\cite{high_RANs} for a generalization to higher dimensions),
where it was proved that the diameter of a RAN is asymptotically bounded above by a constant times the logarithm of the number of vertices.
It was shown in~\cite{define_RANs,RANs_powerlaw} that RANs exhibit a power law degree distribution.
The average distance between two vertices in a typical RAN was shown to be logarithmic by Albenque and Marckert~\cite{RANs_average_distance}.
The degree distribution, $k$ largest degrees and $k$ largest eigenvalues (for fixed $k$) and the diameter
were studied in Frieze and Tsourakakis~\cite{first}. We continue this line of research by
studying the asymptotic properties of the longest (simple) paths in RANs
and giving sharp estimates for the diameter of a typical RAN.

Before stating our main results, we need a few definitions.
In this paper $n$ (respectively, $m$) always denotes the number of vertices (respectively, faces) of the RAN.
All logarithms are in the natural base.
We say an event $A$ happens \emph{asymptotically almost surely (a.a.s.)}
if $\p{A}$ approaches 1 as $n$ goes to infinity.
{For two functions $f(n)$ and $g(n)$ we write $f \sim g$ if
$\lim_{n\rightarrow \infty} { \frac{f(n)}{g(n)} } = 1 \: .$
For a random variable $X = X(n)$ and a function $f(n)$, we say $X$ is \emph{a.a.s.\ asymptotic to} $f(n)$
(and write \emph{a.a.s.\ $X\sim f(n)$}) if for every fixed $\varepsilon>0$,
$$\lim_{n\rightarrow \infty} \p{ f(n) (1-  \varepsilon) \leq {X} \leq f(n) (1+ \varepsilon)} = 1 \:,$$
and we say  \emph{a.a.s.\ $X=o\big(f(n)\big)$}  if
for every fixed $\varepsilon>0$,
$\lim_{n\rightarrow \infty} \p{  {X} \leq \varepsilon f(n) } = 1 \: .$

The authors  of~\cite{first} conjecture in their concluding remarks  that a.a.s.\ a RAN has a path of length $\Omega(n)$.
We refute this conjecture by showing the following theorem.
Let $\mathcal{L}_m$ be a random variable denoting the number of vertices in a longest path
in a RAN with $m$ faces.

\begin{theorem}
\label{thm:longest_upper}
A.a.s.\ we have $\mathcal{L}_m = o(m)$.
\end{theorem}

Recall that a RAN on $n$ vertices has $2n-5$ faces, so Theorem~\ref{thm:longest_upper}
implies that a.a.s.\ a RAN does not have a path of length $\Omega(n)$.

We also prove the following lower bounds for the length of a longest path deterministically, and its expected value in a RAN.

\begin{theorem}
\label{thm:longest_lower}
For every positive integer $m$, the following statements are true.
\begin{itemize}
\item[(a)]
$ \mathcal{L}_m \geq m^{\log 2 / \log 3} + 2 \:.$
\item[(b)]
$\e{\mathcal{L}_m} = \Omega\left ( m^{0.88} \right) \:.$
\end{itemize}
\end{theorem}

The proofs of Theorems~\ref{thm:longest_upper}~and~\ref{thm:longest_lower}
are built on two novel graph theoretic observations, valid for all subgraphs of apollonian networks.

We also study the diameter of RANs.
In~\cite{first} it was shown that the diameter of a RAN is a.a.s.\ at most
$\eta_2 \log n$, where $\eta_2 \approx 7.081$ is the unique solution greater than 1 of $\exp\left ( {1}/{x} \right)  =  {3e}/{x}$.
(Our statement here corrects a minor error in~\cite{first}, propagated from  Broutin and Devroye~\cite{treeheight}, which stated that $\eta_2$ is the unique solution less than 1.)
In~\cite{RANs_average_distance} it was shown that a.a.s.\ the
distance between two randomly chosen vertices of a RAN (which naturally gives a lower bound on the diameter)
is asymptotic to $\eta_1 \log n$, where $\eta_1 = 6/11 \approx 0.545$.
In this paper, we provide the asymptotic value for the diameter of a typical RAN.

\begin{theorem}
\label{thm:diameter}
A.a.s.\ the diameter of a RAN on $n$ vertices is asymptotic to $c \log n$, with
$c=(1-\hat x^{-1})/\log h(\hat x)\approx 1.668$, where
$$
h(x)=\frac{12x^3}{1-2x}- \frac{6x^3}{1-x} \:,
$$
and $\hat x\approx 0.163$ is the unique solution in the interval $(0.1,0.2)$ to
$$
x(x-1)h'(x)=h(x)\log h(x)\:.
$$
\end{theorem}

The proof of Theorem~\ref{thm:diameter} consists of a nontrivial reduction
of the problem of estimating the diameter to the problem of estimating the height of a certain skewed random tree,
which can be done by applying a result of~\cite{treeheight}.

We start with some preliminaries in Section~\ref{sec:preliminaries},
and prove Theorems~\ref{thm:longest_upper},~\ref{thm:longest_lower}, and~\ref{thm:diameter}
in Sections~\ref{sec:longest_upper},~\ref{sec:longest_lower}, and~\ref{sec:diameter}, respectively.

\section{Preliminaries}
\label{sec:preliminaries}
The following result is due to Eggenberger and P\'{o}lya~\cite{polya} (see, e.g., Mahmoud~\cite[Theorem~5.1.2]{urns}).
\begin{theorem}
\label{thm:beta}
Start with $w$ white balls and $b$ black balls in an urn.
In each step, pick a ball uniformly at random from the urn, look at its colour, and return it to the urn; also add $s$ balls of the same colour to the urn.
Let $w_n$ and $t_n$ be the number of white balls and the total number of balls in the urn after $n$ draws.
Then, for any $\alpha\in[0,1]$ we have
$$\lim_{n\rightarrow \infty} \p{\frac{w_n}{t_n} < \alpha} = \frac{\Gamma((w+b)/s)}{\Gamma(w/s)\Gamma(b/s)} \int_{0}^{\alpha} x ^{\frac{w}{s} - 1} (1-x)^{\frac{b}{s}-1} \: \mathrm{d}x \:.$$
\end{theorem}

Note that the right hand side equals $\p{\operatorname{Beta}(w/s,b/s) < \alpha}$,
where $\operatorname{Beta}(p,q)$ denotes a beta random variable with parameters $p$ and $q$.
The urn described in Theorem~\ref{thm:beta} is called the \emph{Eggenberger-P\'{o}lya} urn.

Let $\triangle$ be a triangle.
The \emph{standard 1-subdivision} of $\triangle$ is the set of three triangles obtained from subdividing $\triangle$ once.
For $k>1$, the \emph{standard $k$-subdivision} of $\triangle$ is the set of triangles obtained from
subdividing each triangle in the standard $(k-1)$-subdivision of $\triangle$ exactly once.
In Figure~\ref{fig:path}, the standard 2-subdivision of a triangle is illustrated.

Consider a triangle $\triangle$ containing more than one face in a RAN,
and let $\triangle_1,\triangle_2,\triangle_3$ be the three triangles in its standard 1-subdivision.
We can analyze the number of faces inside $\triangle_1$ by modelling the process of building the RAN
as an Eggenberger-P\'{o}lya urn:
after the first subdivision of $\triangle$, each of $\triangle_1$, $\triangle_2$, and $\triangle_3$
contains exactly one face.
We start with one white ball corresponding to the only face in $\triangle_1$,
and two black balls corresponding to the two faces in $\triangle_2$ and $\triangle_3$.
In each subsequent step, we choose a face uniformly at random, and subdivide it.
If the face is in $\triangle_1$, then the number of faces in $\triangle_1$ increases by 2, and otherwise the number of faces not in $\triangle_1$ increases by 2.
Thus after $k$ subdivisions of $\triangle$, the number of faces in $\triangle_1$ has the same distribution as the number of white balls
in an Eggenberger-P\'{o}lya urn with $w=1$, $b=2$, and $s=2$, after $k-1$ draws.
This observation leads to the following corollary.
\begin{corollary}
\label{cor:fairness}
Let $\triangle$ be a triangle containing $m$ faces in a RAN, and let $Z_1,Z_2,\dots,Z_9$ be the number of
faces inside the 9 triangles in the standard $2$-subdivision of $\triangle$.
Given $\varepsilon>0$, there exists $m_0=m_0(\varepsilon)$ such that
for $m > m_0$,
$$ \p{\min\{Z_1,\ldots,Z_9\} / m <\varepsilon } < 13 \sqrt[4]{\varepsilon} \:.$$
\end{corollary}

\begin{proof}
Let $\overline{\triangle}$ be a triangle containing $\overline{m}$ faces in a RAN, and let $W_1,W_2,W_3$ be the number of faces
inside the three triangles in the standard 1-subdivision of $\overline{\triangle}$.
Say that $\overline{\triangle}$ is \emph{balanced} if
$$\min \{ W_1, W_2, W_3 \} / \overline{m} \geq \sqrt{\varepsilon} \:.$$

By Theorem~\ref{thm:beta}, for a  given  $1\leq i \leq 3$ we have
$$\lim_{\overline{m}\rightarrow \infty} \p{\frac{W_i}{\overline{m}} < \sqrt{\varepsilon}} =
\int_{0}^{\sqrt{\varepsilon}} \frac{\Gamma(3/2)}{\Gamma(1)\Gamma(1/2)}\:x^{{-1}/{2}}\: \mathrm{d}x = \sqrt{\sqrt{\varepsilon}}\:.$$
In particular, there exists $\overline{m}_0$ such that
$$\p{\frac{W_i}{\overline{m}} < \sqrt{\varepsilon}} < \sqrt[4]{1.1 \varepsilon}$$
for $\overline{m} > \overline{m}_0$.

Now, take $m_0 = \overline{m}_0 / \sqrt{\varepsilon}$, and let $\triangle$ be a triangle containing $m>m_0$ faces in a RAN.
The probability that $\triangle$ is balanced is at least $ 1 - 3 \sqrt[4]{1.1 \varepsilon}$ by the union bound.
If $\triangle$ is balanced, then each of the three triangles in the standard 1-subdivision of $\triangle$
contains more than $m_0 \sqrt{\varepsilon} = \overline{m}_0$ faces,
so the probability that a certain one of them is not balanced is at most  $3\sqrt[4]{1.1\varepsilon}$.
Note that if $\triangle$ and these three triangles are balanced, then $\min\{Z_1,\cdots,Z_9\}/m \geq \varepsilon$.
Hence by the union bound,
$$\p{\min\{Z_1,\cdots,Z_9\} / m <\varepsilon } < 12 \sqrt[4]{1.1 \varepsilon} < 13 \sqrt[4]{\varepsilon}\:.\qedhere$$
\end{proof}

We include some definitions here.
Let $G$ be a RAN.
We denote the vertices incident with the unbounded face by $\nu_1,\nu_2,\nu_3$.
All trees we consider are rooted.
We define a tree $T$, called  the \emph{$\triangle$-tree} of $G$, as follows.
There is a one to one correspondence between the triangles in $G$ and the nodes of $T$.
For every triangle  $\triangle$ in $G$,
we denote its corresponding node in $T$ by $\node{\triangle}$.
To build $T$, start with a single root node, which   corresponds to the triangle $\nu_1 \nu_2 \nu_3$ of $G$.
Wherever a triangle $\triangle$ is subdivided into triangles $\triangle_1$, $\triangle_2$, and $\triangle_3$,
generate three children $\node{\triangle_1}$, $\node{\triangle_2}$, and $\node{\triangle_3}$ for $\node{\triangle}$,
and extend the  correspondence in the natural manner.
Note that this is a random ternary tree, with each node having either zero or three children,
and has $3n-8$ nodes and $2n-5$ leaves.
We use the term ``nodes'' for the vertices of $T$, so that ``vertices'' refer to the vertices of $G$.
Note that the leaves of $T$ correspond to the faces of $G$.
The \emph{depth} of a node $\node{\triangle}$ is its distance to the root.

\section{Upper bound for a longest path}
\label{sec:longest_upper}

In this section we prove Theorem~\ref{thm:longest_upper},
stating that a.a.s.\ all paths in a RAN have length $o(n)$.
The set of \emph{grandchildren} of a node is the set of children of its children,
so every node in a ternary tree has between zero and nine grandchildren.
For a triangle $\triangle$ in $G$, $I(\triangle)$ denotes the set of vertices of $G$ that are \emph{strictly inside} $\triangle$.

\begin{lemma}
Let $G$ be a RAN and let $T$ be its $\triangle$-tree.
Let $\node{\triangle}$ be a node of $T$ with nine grandchildren $\node{\triangle_1},\node{\triangle_2},\dots,\node{\triangle_9}$.
Then the vertex set of a path in $G$ does not intersect all of the $I(\triangle_i)$'s.
\label{lem:9children}
\end{lemma}
\begin{proof}
  There are exactly $7$ vertices in the boundaries of the {triangles}
  corresponding to the grandchildren of $\node{\triangle}$.
  Let $v_1,\dotsc, v_7$ denote such
  vertices (see Figure~\ref{fig:path}). Let $P=u_1 u_2\dots u_p$ be a path in $G$.
Clearly, when $P$ enters or leaves one of $\triangle_1,{\triangle_2},\dots,{\triangle_9}$,
it must go through a $v_i$.
So $P$ does not contain vertices from
more than one triangle between two consecutive occurrences of a $v_i$.
Since $P$ goes through each $v_i$ at most once,
the vertices $v_i$ split $P$ up into at most eight sub-paths.
Hence $P$ contains vertices from at most eight of the triangles $\triangle_i$.
  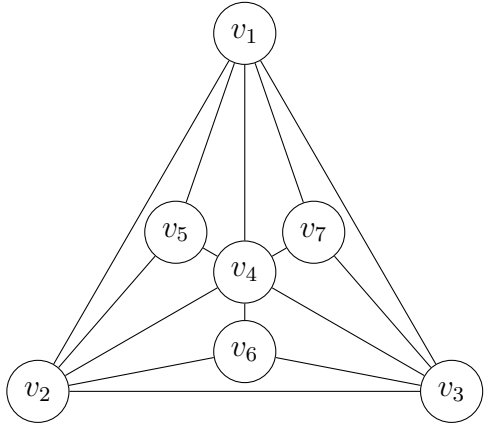
\begin{figure}
    \centering
\begin{tikzpicture}
\def \side {5.5cm}
\node[circle,draw,minimum size=.7cm] at (0,0) (v2) {$v_2$};
\node[circle,draw,minimum size=.7cm] at (60:\side) (v1) {$v_1$};
\node[circle,draw,minimum size=.7cm] at (0:\side) (v3) {$v_3$};
\node[circle,draw,minimum size=.7cm] at (barycentric cs:v1=1,v2=1,v3=1)
(v4) {$v_4$};
\node[circle,draw,minimum size=.7cm] at (barycentric cs:v1=1,v2=1,v4=1)
(v5) {$v_5$};
\node[circle,draw,minimum size=.7cm] at (barycentric cs:v4=1,v2=1,v3=1)
(v6) {$v_6$};
\node[circle,draw,minimum size=.7cm] at (barycentric cs:v1=1,v4=1,v3=1)
(v7) {$v_7$};

\draw (v1)--(v2)--(v3)--(v1);
\draw (v4)--(v1);
\draw (v4)--(v2);
\draw (v4)--(v3);
\draw (v5)--(v1);
\draw (v5)--(v2);
\draw (v5)--(v4);
\draw (v6)--(v4);
\draw (v6)--(v2);
\draw (v6)--(v3);
\draw (v7)--(v1);
\draw (v7)--(v4);
\draw (v7)--(v3);

\end{tikzpicture}
\caption{A triangle in $G$ corresponding to a node of $T$ with $9$
  grandchildren. Vertices $v_1,\dotsc, v_7$ are the vertices
  in the boundaries of the triangles corresponding to these
  grandchildren.}
\label{fig:path}
\end{figure}
\end{proof}

We first sketch a proof of Theorem~\ref{thm:longest_upper}. Let $G$ be a RAN on $n$ vertices, and let $T$ be its $\triangle$-tree.
The 2-subdivision of the triangle $\nu_1 \nu_2 \nu_3$ consists of nine triangles,
and every path misses the vertices in at least one of them by Lemma~\ref{lem:9children}.
We can now apply the same argument inductively for the other eight triangles, and repeat.
Note that if the distribution of vertices in the nine triangles of every 2-subdivision were always moderately balanced, this argument would immediately prove the theorem (by extending it to $O(\log n)$ depth).
Unfortunately, the distribution is biased towards becoming unbalanced:
the greater the number of vertices falling in a certain triangle, the higher the probability that the next vertex falls in the same triangle.
However, Corollary~\ref{cor:fairness} gives an upper bound for the probability
that this distribution is very unbalanced.
The idea is to use this Corollary iteratively and to use independence of events cleverly  to bound the probability of certain ``bad'' events.

It is easy to see that $T$ is a random ternary tree on $3n-8$ nodes
in the sense of Drmota~\cite{random_trees}.
The following theorem is due to Chauvin and Drmota~\cite[Theorem~2.3]{saturation}
(we use the wording of~\cite[Theorem~6.47]{random_trees}).

\begin{theorem}
\label{thm:saturation}
Let $\overline{H}_n$ denote the largest number $L$
such that a random $n$-node ternary tree has precisely $3^L$ nodes at depth $L$.
Let $\psi \approx 0.152$ be the unique solution in $(0,3)$ to
$$ 2 \psi \log \left (\frac{3e}{2\psi}\right) = 1 \:.$$
Then we have
$$\e{\overline{H}_n} \sim \psi \log n \:,$$
and there exists a constant $\kappa>0$ such that for every $\varepsilon>0$,
$$\p{|\overline{H}_n - \e{\overline{H}_n}| > \varepsilon} = O(\exp(-\kappa\varepsilon)) \:.$$
\end{theorem}

Let $D = 0.07 \log n$.
Then, the following is obtained immediately.

\begin{corollary}
\label{cor:standard}
A.a.s.\ there are $3^{2D}$ nodes at depth $2D$ of $T$.
\end{corollary}

Let $\varepsilon>0$ be a fixed number such that $3 (13\sqrt[4]{4\varepsilon})^{1/5} < 1$,
and let $p_F = 1 - 13\sqrt[4]{4\varepsilon}$.
Notice that $3 (1-p_F)^{1/5} < 1$.
We say node $\node{\triangle}$ is \emph{fair} if at least one of the following holds:
\begin{enumerate}
\item[(i)]  the number of faces inside $\triangle$ is less than $3^D$, or
\item[(ii)] $\node{\triangle}$ has nine grandchildren $\node{\triangle_1},\node{\triangle_2},\dots,\node{\triangle_9}$,
and $|I(\triangle_i)| \geq \varepsilon |I(\triangle)|$ for all $1\leq i \leq 9$.
\end{enumerate}
A triangle $\triangle$ in $G$ is fair if its corresponding node $\node{\triangle}$ is fair.

\begin{lemma}
\label{lem:fairness}
Let $\node{\triangle}$ be a node in $T$  with nine grandchildren, and let $U$ be a subset of the set of ancestors of $\node{\triangle}$,
not including the parent of $\node{\triangle}$.
The probability that $\node{\triangle}$ is fair, conditional on
all nodes in $U$ being unfair, is at least $p_F$.
\end{lemma}

\begin{proof}
Let $n$ be sufficiently large that $3^D > m_0(2\varepsilon)$, where $m_0(2\varepsilon)$ is defined as in
Corollary~\ref{cor:fairness}.
Let $\overline{M}$ denote the number of faces inside $\triangle$,
and let $\overline{m} \geq 3^D$ be arbitrary.
If $\overline{M} < 3^D$, then $\node{\triangle}$ is fair  by definition,
so it is enough to prove that
$$\p{\node{\triangle}\mathrm{\ is\ fair} \cond \mathrm{nodes\ in\ }U\mathrm{\ are\ unfair}, \overline{M} = \overline{m}} \geq p_F \:.$$
Since $U$ does not contain the parent of $\node{\triangle}$,
conditional on nodes in $U$ being unfair and $\overline{M}=\overline{m}$,
the subgraph of  $G$ induced by vertices on and inside $\triangle$
is distributed as a RAN with $\overline{m}$ faces.

Let $\triangle_1,\dots,\triangle_9$ be the nine triangles in the standard 2-subdivision of $\triangle$,
and let $Z_1,Z_2,\dots,Z_9$ be the number of faces inside them.
By Corollary~\ref{cor:fairness} and since $\overline{m} > m_0(2\varepsilon)$, with probability at least $p_F$
for all $1\leq i\leq 9$,
$$Z_i \geq 2 \varepsilon \overline{m} \:,$$
and so
$$I(\triangle_i) = \frac{Z_i - 1}{2} \geq \frac{2\varepsilon \overline{m} - 1}{2} \geq \varepsilon\:\frac{\overline{m}-1}{2} = \varepsilon I(\triangle) \:, $$
which implies that $\node{\triangle}$ is fair.
\end{proof}

Let $k = (\log \log n) / 2$.
Let $d_0 = 0$ and $d_i = 2^{i-1} k$ for $1\leq i \leq k$.
Notice that $d_k < D$.
\begin{lemma}
\label{lem:lots_of_fair}
A.a.s\ the following is true.
Let $v$ be an arbitrary node of $T$ at depth $d_i$ for some $1\leq i \leq k$,
and let $u$ be the ancestor of $v$ at depth $d_{i-1}$.
Then there is at least one fair node $f$ on the $(u,v)$-path in $T$,
such that the depth of $f$ is between $d_{i-1}$ and $d_i - 2$, inclusive.
\end{lemma}

\begin{proof}
Let us say that a node is \emph{bad} if the conclusion of the lemma is false for it.
We prove that the probability that a bad node exists is $o(1)$.
Let $v$ be a node at depth $d_i$ and $u$ be its ancestor at depth $d_{i-1}$.
Let $x_0 = v, x_1, x_2, \dots, x_r = u$ be the $(v,u)$-path in $T$, where $r = d_i - d_{i-1}$.
By Lemma~\ref{lem:fairness}, the probability that none of
$x_{2\lfloor r/2 \rfloor},x_{2\lfloor r/2 \rfloor - 2}, \dots, x_4, x_2$ is fair is at most
$$(1-p_F)^{\lfloor r/2 \rfloor} \leq (1-p_F)^{(d_i - d_{i-1} - 1)/2} \leq (1-p_F)^{d_{i}/5}\:.$$
There are at most $3^{d_i}$ nodes at depth $d_i$, so by the union bound, the probability that there is at least one bad node $v$ at depth $d_i$
is at most
$$3^{d_i} (1-p_F)^{d_i/5} = \left [ 3 (1-p_F)^{1/5} \right ]^{d_i} \leq \left [ 3 (1-p_F)^{1/5} \right ]^{k} = o(1/k) \:,$$
by the definition of $d_i$  and as
$3 (1-p_F)^{1/5} < 1$ and $k\to\infty$.
Consequently, the probability that there exists a bad node
whose depth lies in $\{d_1,d_2,\dots,d_k\}$ is $o(1)$.
\end{proof}

We are now ready to prove Theorem~\ref{thm:longest_upper}.

\begin{proof}[Proof of Theorem~\ref{thm:longest_upper}.]
Let $G$ be a RAN with $n$ vertices and $m$ faces, and let $T$ be the $\triangle$-tree of $G$.
The \emph{depth} of a vertex $v$ of $G$ is defined as
$\max \{\operatorname{depth}(\triangle) : v \in I(\triangle) \}$,
and we define the depth of $\nu_1, \nu_2, \nu_3$ to be $-1$.
Say a vertex is \emph{deep} if its depth is greater than $D$, and is \emph{shallow} otherwise.
Let $n_D$ denote the number of deep vertices.
Note that the number of shallow vertices is at most $(3^{D+1}+5)/2$,
which is $o(n)$ by the choice of $D$, so $n_D = n - o(n)$.
For a node $\node{\triangle}$ of $T$, let $I_D(\triangle)$ be the set of deep vertices in $I(\triangle)$, and
for a subset $A$ of nodes of $T$, let
$$I_D(A) = \bigcup_{\node{\triangle}\in A} I_D(\triangle) \:.$$

\begin{claim}
If $T$ is full down to depth $2D$ where $D\ge 2$,
then any fair node $\node{\triangle}$ with depth at most $D$
has nine grandchildren $\node{\triangle_1},\node{\triangle_2},\dots,\node{\triangle_9}$
such that
\begin{equation}
\label{eq:farisplit}
|I_D(\triangle_i)| \geq \varepsilon |I_D(\triangle)| / 2 \qquad i=1,2,\dots,9.
\end{equation}
\end{claim}

\begin{claimproof}
Assume that $T$ is full down to depth $2D$.
We first show that for any triangle $\triangle$,
\begin{equation}
\label{eq:half}
|I_D(\triangle)| \geq |I(\triangle)| / 2 \:.
\end{equation}
To prove (\ref{eq:half}), let $\triangle$ be a triangle at depth $r$.
If $r>D$, then $I(\triangle)$ contains no shallow vertices, and (\ref{eq:half}) is obviously true.
Otherwise, the number of shallow vertices in $I(\triangle)$
equals $1 + 3 + \dots + 3^{D-r} = (3^{D-r+1}-1)/2$,
whereas the number of vertices in $I(\triangle)$ at depth $D+1$
equals $3^{D-r+1}$, where we have used the fact that $T$ is full down to depth at least $D+2$.
Thus $I(\triangle)$ contains more deep vertices than shallow vertices,
and (\ref{eq:half}) follows.

Now, let $\node{\triangle}$ be a fair triangle having depth at most $D$.
Since $T$ is full down to depth $2D$, {the number of faces inside $\triangle$ is at least $3^D$}.
So, as $\node{\triangle}$ is fair, it has nine grandchildren $\node{\triangle_1},\node{\triangle_2},\dots,\node{\triangle_9}$
such that $|I(\triangle_i)| \geq \varepsilon |I(\triangle)|$ for all $1\leq i \leq 9$.
Applying (\ref{eq:half}) gives
\begin{equation*}
|I_D(\triangle_i)| \geq |I(\triangle_i)|/ 2 \geq \varepsilon |I(\triangle)| / 2 \geq \varepsilon |I_D(\triangle)| / 2 \qquad i=1,2,\dots,9 \:,
\end{equation*}
as required.
\end{claimproof}

We {may} condition on two events that happen a.a.s.:\ the first one is the conclusion of Lemma~\ref{lem:lots_of_fair},
and the second one is that of Corollary~\ref{cor:standard}, namely that
$T$ is full down to depth $2D$.

To complete the proof of the theorem, for a given path $P$ in $G$, we will define a sequence $B_0, B_1, \dots, B_k$ of sets of nodes of $T$,
such that for all $0 \leq i \leq k$ we have
\begin{enumerate}
\item [(i)] $ |I_D(B_i)| \geq n_D \left(1 - \left(1-\frac{\varepsilon}{2}\right)^i\right)$, and
\item [(ii)] $V(P) \cap I_D(B_i) = \emptyset$.
\end{enumerate}
Before defining the $B_i$'s, let us show that this completes the proof.
Notice that (i) gives
$$|I_D(B_k)| \geq n_D - n_D (1-\varepsilon/2)^k \geq n_D - n_D \exp(-\varepsilon k /2)\:,$$
which is $n - o(n)$ since $n_D = n-o(n)$ and $\varepsilon k = \omega(1)$.
Therefore, by (ii),
$$|V(P)| \leq |V(G) \setminus I_D(B_k)| = o(n)\:.$$

{So, now} we define the sets $B_i$.
Let $S_i$ denote the set of nodes of $T$ at depth $d_i$.
Let $B_0 = \emptyset$ and we define the $B_i$'s inductively,
in such a way that $B_i \subseteq S_i$.
Fix  $1 \leq i \leq k$, and assume that $B_{i-1}$ has already been defined.
Let $C_i$ be the set of nodes at depth $d_i$ whose ancestor at depth $d_{i-1}$ is in $B_{i-1}$
(so, in particular, $C_1=\emptyset$).
By the induction hypothesis, $V(P)$ does not intersect $I_D(B_{i-1}) = I_D(C_i)$,
and $ |I_D(C_i)| = |I_D(B_{{i-1}})| \geq n_D \left(1 - (1-\varepsilon/2)^{i-1}\right)$.

Since the conclusion of Lemma~\ref{lem:lots_of_fair} is true, there exists a set $F$ of fair nodes,
with depths between $d_{i-1}$ and $d_i - 2$,
such that every $v \in S_i \setminus C_i$ is a descendent of some node in $F$.
Now, for every $x,y \in F$ such that $y$ is a descendent of $x$,
remove $y$ from $F$.
This results in a set $\{u_1,u_2,\dots,u_s\}$ of fair nodes,
with depths between $d_{i-1}$ and $d_i - 2$,
such that every $v \in S_i \setminus C_i$ is a descendent of a unique $u_j$.
Recall that $d_k < D$ and so all the $u_j$'s have depths less than $D$.

Let $w_1,\dots,w_9$ be the grandchildren of $u_1$.
By Lemma~\ref{lem:9children}, $V(P)$ does not intersect all of the $I(w_i)$'s{;} say it does not intersect $I({w_1})$.
Then mark all of the descendants of $w_1$, and perform a similar procedure for $u_2,\dots,u_s$.
Let $M_i$ be the set of marked nodes in $S_i$. See Figure~\ref{fig:induction}.
\begin{figure}
\label{fig:induction}
\begin{center}
    \begin{tikzpicture}[dot/.style={fill=black,circle,minimum
        size=3pt},scale=.5]
\usetikzlibrary{shapes}
\draw (2,0) ellipse (1.5 and .5);
\node (a) at (1,0) {$\times$};
\node (b) at (2,0) {$\bullet$};
\node (c) at (3,0) {$\bullet$};
\node (e1) at (0.5,0) {};
\node (f1) at (3.5,0) {};

\draw (7.5,0) ellipse (3 and .5);
\node (d) at (5,0) {$\bullet$};
\node (e) at (6,0) {$\bullet$};
\node (dd) at (9,0) {$\bullet$};
\node (e2) at (10,0) {$\bullet$};
\node (f) at (7,0) {$\times$};
\node (g) at (8,0) {$\times$};
\node (e2) at (4.5,0) {};
\node (f2) at (10.5,0) {};

\draw (14.5,0) ellipse (3 and .5);
\node (h) at (12,0) {$\times$};
\node (i) at (13,0) {$\times$};
\node (j) at (14,0) {$\bullet$};
\node (k) at (15,0) {$\bullet$};
\node (l) at (16,0) {$\bullet$};
\node (m) at (17,0) {$\bullet$};
\node (e3) at (11.5,0) {};
\node (f3) at (17.5,0) {};

\draw (20.5,-0.5) ellipse (2 and 1.25);
\node (n) at (19,0) {$+$};
\node (o) at (20,0) {$+$};
\node (p) at (21,0) {$+$};
\node (q) at (22,0) {$+$};
\node (r) at (19,-1) {$+$};
\node (s) at (20,-1) {$+$};
\node (t) at (21,-1) {$+$};
\node (u) at (22,-1) {$+$};
\node (e4) at (18.5,-0.5) {};
\node (f4) at (22.5,-0.5) {};

\draw  (0,-2) rectangle (23,1);
\node (v) at (0,-.5) [label=180:$S_i$] {};

\node (u1) at (2,3) {$u_1$};
\node (u2) at (7.5,3) {$u_2$};
\node (u3) at (14.5,3) {$u_3$};

\draw (u1)--(e1);
\draw (u1)--(f1);
\draw (u2)--(e2);
\draw (u2)--(f2);
\draw (u3)--(e3);
\draw (u3)--(f3);

\draw (20.5,3) ellipse (1.5 and .5);
\node (u41) at (19,3) {};
\node (u42) at (22,3) [label={0:$B_{i-1}$}] {};
\draw (u41)--(e4);
\draw (u42)--(f4);

\end{tikzpicture}
\end{center}
\caption{Illustration for the inductive step in the proof of Theorem~\ref{thm:longest_upper}:
Vertices in $C_i$ are shown as $+$,
vertices in $M_i$ are shown as $\times$,
and vertices in $S_i \setminus (C_i \cup M_i)$ are shown as dots.}
\end{figure}
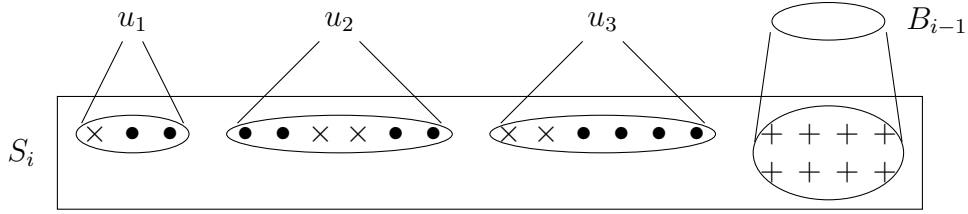
Thus $V(P) \cap I_D(M_i) = \emptyset$.
Moreover, since the $u_j$'s are fair
and the $I_D(u_j)$'s are disjoint, it follows from the claim that
$$|I_D(M_i)| \geq \sum_{j=1}^{s} \varepsilon |I_D(u_j)| / 2 = {\varepsilon} | I_D(S_i \setminus C_i) | / 2 = \varepsilon (n_D - |I_D(B_{i-1})|)  / 2\:.$$
Now, let $B_i = C_i \cup M_i$. Then we have
\begin{align*}
|I_D(B_i)| &= |I_D(C_i)| + |I_D(M_i)|\\
 & \geq \: |I_D(B_{i-1})| + \frac{\varepsilon}2 \Big(n_D - |I_D(B_{i-1})|\Big)
= |I_D(B_{i-1})| (1-\frac{\varepsilon}{2}) +  \frac{\varepsilon n_D}{2} \\
& \geq n_D \left(1 - \left(1-\frac{\varepsilon}{2}\right)^{i-1}\right) \left(1-\frac{\varepsilon}{2}\right) + \frac{\varepsilon n_D}{2}
= n_D \left(1 - \left(1-\frac{\varepsilon}{2}\right)^i\right)\:,
\end{align*}
and $V(P)$ does not intersect $I_D(B_i)$.
\end{proof}

\begin{remark}
Noting that $n - n_D < n^{1-\delta}$ for some fixed $\delta>0$
and being more careful in the calculations above shows that indeed
we have a.a.s.\ $\mathcal{L}_m \leq n \left(\log n\right)^{-\Omega(1)}$.
\end{remark}

\section{Lower bounds for a longest path}
\label{sec:longest_lower}

In this section we prove Theorem~\ref{thm:longest_lower}.
We first prove part (a), i.e., we give a deterministic lower bound for the length of a longest path in a RAN.
Recall that $\mathcal{L}_m$ denotes the number of vertices of a longest path in a RAN with $m$ faces.
 Let $G$ be a RAN with $m$ faces, and
let $v$ be the unique vertex that is adjacent to $\nu_1$, $\nu_2$, and $\nu_3$.
For $1\leq i \leq 3$, let $\triangle_i$ be the triangle with vertex set $\{v,\nu_1,\nu_2,\nu_3\} \setminus \{\nu_i\}$.
Define the random variable $\mathcal{L}'_m$ as the largest number $L$
such that for every permutation $\pi$ on $\{1,2,3\}$,
there is a path in $G$ of  $L$ edges from $\nu_{\pi(1)}$ to $\nu_{\pi(2)}$ not containing $\nu_{\pi(3)}$.
Clearly we have $\mathcal{L}_m\ge \mathcal{L}'_m  +2$.

\begin{proof}[Proof of Theorem~\ref{thm:longest_lower}(a).]
Let $\xi = \log 2 / \log 3$.
We prove by induction on $m$ that $\mathcal{L}'_m \geq m^{\xi}$.
This is obvious for $m=1$, so assume that $m>1$.
Let $m_i$ denote the number of faces in $\triangle_i$.
Then $m_1 + m_2 + m_3 = m$.
By symmetry, we may assume that $m_1 \geq m_2 \geq m_3$.
For any given $1\leq i\leq 3$, it is easy to find a path avoiding $\nu_i$
that connects the other two $\nu_j$'s by attaching two appropriate paths
in $\triangle_1$ and $\triangle_2$ at vertex $v$. (See Figures~\ref{fig:pathmerge}(a)--(c).)
By the induction hypothesis, these paths can be chosen to have lengths at least ${m_1}^{\xi}$ and ${m_2}^{\xi}$, respectively.
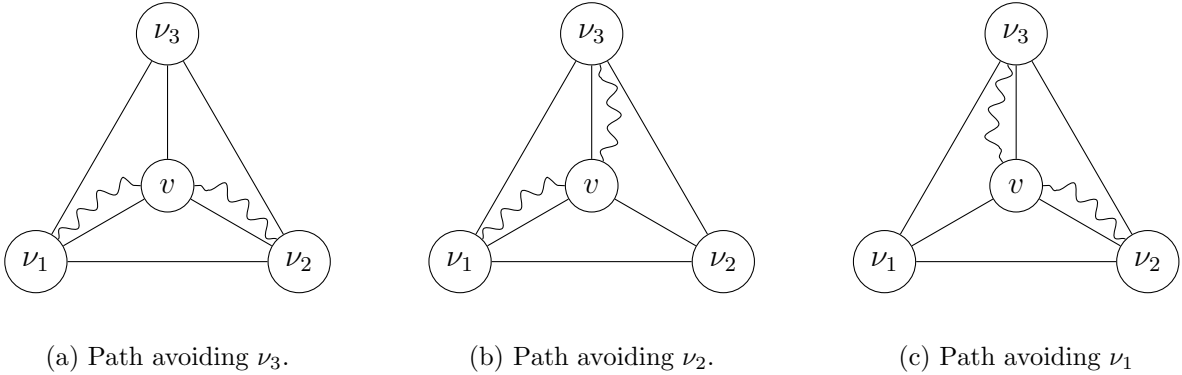
\begin{figure}
\centering
\subfigure[Path avoiding $\nu_3$.]{
\begin{tikzpicture}
\def \side {3.5cm}
\node[circle,draw,minimum size=.7cm] at (0,0) (nu1) {$\nu_1$};
\node[circle,draw,minimum size=.7cm] at (60:\side) (nu3) {$\nu_3$};
\node[circle,draw,minimum size=.7cm] at (0:\side) (nu2) {$\nu_2$};
\node[circle,draw,minimum size=.7cm] at (barycentric cs:nu1=1,nu2=1,nu3=1)
(v) {$v$};

\draw (nu1)--(nu2)--(nu3)--(nu1);
\draw (v)--(nu1);
\draw (v)--(nu2);
\draw (v)--(nu3);
\draw[decoration={snake},
         decorate] (nu1) to[out=45,in=180] (v);
\draw[decoration={snake},
         decorate] (nu2) to[out=135,in=0] (v);
\end{tikzpicture}
}
\hspace{.75cm}
\subfigure[Path avoiding $\nu_2$.]{
\begin{tikzpicture}

\def \side {3.5cm}

\node[circle,draw,minimum size=.7cm] at (0,0) (nu1) {$\nu_1$};
\node[circle,draw,minimum size=.7cm] at (60:\side) (nu3) {$\nu_3$};
\node[circle,draw,minimum size=.7cm] at (0:\side) (nu2) {$\nu_2$};
\node[circle,draw,minimum size=.7cm] at (barycentric cs:nu1=1,nu2=1,nu3=1) (v) {$v$};

\draw (nu1)--(nu2)--(nu3)--(nu1);
\draw (v)--(nu1);
\draw (v)--(nu2);
\draw (v)--(nu3);

\draw[decoration={snake},
         decorate] (nu1) to[out=45,in=180] (v);
\draw[decoration={snake},
         decorate] (v) to[out=65,in=-75] (nu3);
\end{tikzpicture}
}
\hspace{.75cm}
\subfigure[Path avoiding $\nu_1$]{
\begin{tikzpicture}

\def \side {3.5cm}

\node[circle,draw,minimum size=.7cm] at (0,0) (nu1) {$\nu_1$};
\node[circle,draw,minimum size=.7cm] at (60:\side) (nu3) {$\nu_3$};
\node[circle,draw,minimum size=.7cm] at (0:\side) (nu2) {$\nu_2$};
\node[circle,draw,minimum size=.7cm] at (barycentric cs:nu1=1,nu2=1,nu3=1) (v) {$v$};

\draw (nu1)--(nu2)--(nu3)--(nu1);
\draw (v)--(nu1);
\draw (v)--(nu2);
\draw (v)--(nu3);

\draw[decoration={snake},
         decorate] (nu3) to[out=-105,in=120] (v);
\draw[decoration={snake},
         decorate] (nu2) to[out=135,in=0] (v);
\end{tikzpicture}
}

\caption{Paths avoiding $\triangle_3$ and one of  the  $\nu_i$'s.}
\label{fig:pathmerge}
\end{figure}
Hence for every permutation $\pi$ of $\{1,2,3\}$,
there is a path from $\nu_{\pi(1)}$ to $\nu_{\pi(2)}$ avoiding $\nu_{\pi(3)}$ with length at least
\begin{equation}
\label{eq:m1+m2}
{m_1}^{\xi} + {m_2}^{\xi}\:.
\end{equation}
It is easily verified that since $m_1 \geq m_2 \geq m_3$ and $m_1 + m_2 + m_3 = m$,
the minimum of (\ref{eq:m1+m2}) happens when $m_1 = m_2 = m / 3$, thus
$$\mathcal{L}'_m \geq {m_1}^{\xi} + {m_2}^{\xi} \geq 2 \left(\frac{m}{3}\right)^{\xi} = m^{\xi}\:,$$
and the proof is complete.
\end{proof}

Next, we use the same idea to give a larger lower bound for $\e{\mathcal{L}_m}$.
Let the random variable $X_i$ denote the number of faces in $\triangle_i$.
Then the $X_i$'s have the same distribution and are not independent.
 It follows from Theorem~\ref{thm:beta} that  as $m$ grows, the distribution of $\frac{X_i}{m}$ converges pointwise
to that of  $\operatorname{Beta}(1/2,1)$.
Moreover, for any fixed $\varepsilon \in [0,1)$, if we condition on $X_1 = \varepsilon m$,
then the subdividing process inside $\triangle_2$ and $\triangle_3$ can be modelled as an Eggenberger-P\'{o}lya urn again,
and it follows from Theorem~\ref{thm:beta} that the distribution of $\frac{X_2}{(1-\varepsilon)m}$
conditional on $X_1 = \varepsilon m$
converges pointwise to that of
 {$\operatorname{Beta}(1/2,1/2)$}.
 Namely, for any fixed $\varepsilon \in [0,1)$ and $\delta \in [0,1]$,
\begin{equation}
\label{eq:secondtriangledistribution}
\lim_{m\rightarrow \infty} \p{\frac{X_2}{(1-\varepsilon)m} \leq \delta {\left\vert\vphantom{\frac{1}{1}}\right.} X_1 = \varepsilon m}
= \int_{0}^{\delta} \frac{\Gamma(1)}{\Gamma({1}/{2})^2} \: x^{-1/2} (1-x)^{-1/2}\: \mathrm{d}x \:.
\end{equation}

We are now ready to prove part (b) of Theorem~\ref{thm:longest_lower}.

\begin{proof}[Proof of Theorem~\ref{thm:longest_lower}(b).]
Let $\zeta = 0.88$.
We prove that there exists a constant $\kappa>0$ such that
$\e{\mathcal{L}'_m} \ge \kappa m^{\zeta}$ holds for all $m\ge 1$.
We proceed by induction on $m$, with the induction base being $m=m_0$,
where $m_0$ is a sufficiently large constant, to be determined later.
By choosing $\kappa$ sufficiently small, we may assume $\e{\mathcal{L}'_m} \ge \kappa m^{\zeta}$ for all $m\leq m_0$.

For $1\leq i \leq 3$, let $X_i$ denote the number of faces  in $\triangle_i$.
Define  a   permutation $\sigma$   on $\{1,2,3\}$ such that
$X_{\sigma(1)}\ge X_{\sigma(2)}\ge X_{\sigma(3)}$, breaking ties randomly.
Then $\sigma$ is a random permutation determined by the $X_i$ and the random choice in the tie-breaking.
By symmetry,  for every fixed $\sigma'\in S_3$, $\p{\sigma=\sigma'}= 1/6$.
From the proof of part (a), we know
$$
\mathcal{L}'_m\ge \mathcal{L}'_{X_{\sigma(1)}}+\mathcal{L}'_{X_{\sigma(2)}}.
$$
Taking the expectation on both sides, we have
\begin{equation}
\e{\mathcal{L}'_m}\ge \e{\mathcal{L}'_{X_{\sigma(1)}}+\mathcal{L}'_{X_{\sigma(2)}}}\ge 6\e{(\mathcal{L}'_{X_{1}}+\mathcal{L}'_{X_{2}})\mathds{1}_{X_1> X_2> X_3}}\:,\label{eq:uniformProb}
\end{equation}
where the second inequality holds by symmetry and as $\p{\sigma=(1,2,3)}= 1/6$.
By the induction hypothesis, for every $x_1,x_2<m$,
$$
\e{\mathcal{L}'_{X_{1}}\mid X_1=x_1}\ge \kappa x_1^{\zeta}, \mathrm{\ and\ }\e{\mathcal{L}'_{X_{2}}\mid X_2=x_2}\ge \kappa x_2^{\zeta}.
$$
Hence,
\begin{equation}
\e{(\mathcal{L}'_{X_{1}}+\mathcal{L}'_{X_{2}})\mathds{1}_{X_1> X_2> X_3}}\ge \kappa\e{(X_1^{\zeta}+X_2^{\zeta})\mathds{1}_{X_1> X_2> X_3}}.\label{eq:induct}
\end{equation}
Let $f_1(x)$ and $f_2(x)$ denote the probability density functions
of $\operatorname{Beta}(1/2,1)$ and $\operatorname{Beta}(1/2,1/2)$, respectively.
Namely,
$$
f_1(x)=\frac{\Gamma(3/2)}{\Gamma(1)\Gamma(1/2)}\:x^{{-1}/{2}}\mathrm{\ and\ } f_2(x)=\frac{\Gamma(1)}{\Gamma({1}/{2})^2} \: x^{-1/2} (1-x)^{-1/2}\:.
$$
Then  it follows from Theorem~\ref{thm:beta} that for any fixed $0\leq t < 1$,
$$
\lim_{m\rightarrow \infty} \p{\frac{X_1}{m} \leq t}
= \int_{0}^{t} f_1(x)\: \mathrm{d}x \:,
$$
and for any fixed $0\leq s \leq 1$, by (\ref{eq:secondtriangledistribution}),
$$
\lim_{m\rightarrow \infty} \p{\frac{X_2}{m} \leq (1-t)s \cond X_1 = t m}=\lim_{m\rightarrow \infty} \p{\frac{X_2}{(1-t)m} \leq s \cond \frac{X_1}{m} = t }
= \int_{0}^{s} f_2(x)\: \mathrm{d}x \:.
$$
 Hence (see Billingsley~\cite[Theorem 29.1 (i)]{ref.Billingsley})
\begin{align*}
\e{\left ( \left(\frac{X_1}{m}\right)^{\zeta} + \left(\frac{X_2}{m}\right) ^{\zeta} \right) \mathds{1}_{X_1 >X_2 > X_3}}& \rightarrow
\int \limits_{t=1/3}^{1} \! \! \int \limits_{s=1/2}^{\min\left\{1, \frac{t}{1-t}\right\}}\!  \left[ t^{\zeta} + (s(1-t))^{\zeta}\right]   f_1(t)f_2(s)\,\mathrm{d}s\,\mathrm{d}t\,,
\end{align*}
as $m\to\infty$.
By the choice of $\zeta$, we have
$$\int_{t=1/3}^{1} \int_{s=1/2}^{\min\{1, \frac{t}{1-t}\}}  \left[t^{\zeta} + (s(1-t))^{\zeta}\right]  f_1(t)f_2(s)   \:\mathrm{d}s\:\mathrm{d}t > 1/6 \:.$$
Then, by~(\ref{eq:uniformProb}) and~(\ref{eq:induct}),
$$
\e{\mathcal{L}'_m}\ge 6\kappa\e{(X_1^{\zeta}+X_2^{\zeta})\mathds{1}_{X_1> X_2> X_3}}> \kappa m^{\zeta},
$$
if we choose $m_0$ sufficiently large.
\end{proof}

\section{Diameter}
\label{sec:diameter}
As mentioned in the introduction,  prior to this work
it had been known that a typical RAN has logarithmic diameter,
and asymptotic lower and upper bounds for the diameter had been proved,
but the asymptotic value had not been determined.
In this section we prove Theorem~\ref{thm:diameter},
which states that a.a.s.\ the diameter of a RAN is asymptotic to $c \log n$,
where $c \approx 1.668$ is the solution of an explicit equation.

Let $G$ be a RAN with $n$ vertices, and recall that
$\nu_1$, $\nu_2$, and $\nu_3$ denote the vertices incident with the unbounded face.
For a vertex $v$ of $G$, let $\tau(v)$ be the minimum graph distance of $v$ to the boundary, i.e.,
$$\tau(v) = \min \{\dist(v,\nu_1),\dist(v,\nu_2),\dist(v,\nu_3)\} \:.$$
The \emph{radius} of $G$ is defined as the maximum of $\tau(v)$ over all vertices $v$.

\begin{lemma}
\label{lem:radius}
Let
$$ h(x)=\frac{12x^3}{1-2x} - \frac{6x^3}{1-x} \:,$$
and let $\hat x$ be the unique solution in $(0.1,0.2)$ to
$$x(x-1)h'(x)=h(x)\log h(x) \:.$$
Finally, let
$$c = \frac { 1-\hat x^{-1}} {\log h(\hat x)} \approx 1.668 \:.$$
Then the radius of $G$ is a.a.s.\ asymptotic to $c \log n / 2$.
\end{lemma}

We first show that this lemma implies Theorem~\ref{thm:diameter}.

\begin{proof} [Proof of Theorem~\ref{thm:diameter}]
 Let $\triangle_1$, $\triangle_2$, and $\triangle_3$ be the three triangles in the standard 1-subdivision of the triangle $\nu_1\nu_2\nu_3$,
and let $n_i$ be the number of vertices on and inside $\triangle_i$.
Let $\diam(G)$ denote the diameter of $G$.
Fix arbitrarily small $\varepsilon, \delta > 0$.
We show that
with probability at least $1-2\delta$ we have
$$(1-\varepsilon) c \log n \leq \diam(G) \leq (1+\varepsilon) c\log n \:.$$
{Here and in the following, we assume $n$ is sufficiently large.}

Let $M$ be a positive integer sufficiently large that, for  a given $1\leq i\leq 3$,
$$ \p{\frac{n_i}{n} < \frac{1}{M}} < \delta / 6 \:.$$
Such an $M$ exists by Theorem~\ref{thm:beta} and the discussion after it.
Let $A$ denote the event
$$ \min\left\{ \frac{n_i}{n} : 1\leq i \leq 3 \right\} \geq \frac{1}{M} \:.$$
By the union bound, $\p{A} \geq 1- \delta / 2$.
We condition on values $(n_1,n_2,n_3)$ such that $A$ happens.
 Note that we have $\log n_i = \log n - O(1)$ for each $i$.

For a triangle $\triangle$, $V(\triangle)$ denotes the three vertices of $\triangle$.
Note that for $1\leq i \leq 3$, the subgraph induced by vertices on and inside $\triangle_i$
is  distributed as  a RAN $G_i$ with $n_i$ vertices.
Hence by Lemma~\ref{lem:radius} and  the union bound,
with probability at least $1-\delta /2$,
the radius of  each of $G_1$, $G_2$ and $G_3$ is at least $(1-\varepsilon) c \log n / 2$.
Hence, with probability at least $1-\delta /2$
there exists $u_1 \in V(G_1)$ with distance at least
$(1-\varepsilon) c \log n / 2$ to $V(\triangle_1)$,
and also there exists $u_2 \in V(G_2)$ with distance at least
$(1-\varepsilon) c \log n / 2$ to $V(\triangle_2)$.
Since any $(u_1,u_2)$-path must contain a vertex from $V(\triangle_1)$ and $V(\triangle_2)$,
with probability at least $1 - \delta/2$,
there exists $u_1, u_2 \in V(G)$ with distance at least $2 (1-\varepsilon) c \log n / 2$,
which implies
$$\p{\diam(G) \geq  c(1-\varepsilon) \log n} \geq
\p{\diam(G) \geq  c (1-\varepsilon) \log n \vert A} \p{A}
> 1 - \delta \:.$$

For the upper bound, let $R$ be the radius of $G$.
Notice that the distance between any vertex and  $\nu_1$ is at most $R+1$,
so $\operatorname{diam}(G)\le 2R+2$.
By Lemma~\ref{lem:radius}, with probability at least $1 - \delta$ we have
$R \le (1+\varepsilon/2)c \log n / 2$. If this event happens, then
$\operatorname{diam}(G) \le (1+\varepsilon)c \log n$.
\end{proof}

The rest of this section is devoted to the proof of Lemma~\ref{lem:radius}.
Let $T$ be the  $\triangle$-tree of $G$, as defined in Section~\ref{sec:preliminaries}.
We categorize the triangles in $G$ into three types.
Let $\triangle$ be a triangle in $G$ with vertex set $\{x,y,z\}$,
and assume that $\tau(x) \leq \tau(y) \leq \tau(z)$.
Since $z$ and $x$ are adjacent, we have $\tau(z) \leq \tau(x) + 1$.
So, ${\triangle}$ can be categorized to be of one of the following types:
\begin{enumerate}
\item if $\tau(x) = \tau(y) = \tau(z)$, then say ${\triangle}$ is of type 1.
\item If $\tau(x) = \tau(y) < \tau(y) + 1 = \tau(z)$, then  say ${\triangle}$ is of type 2.
\item If $\tau(x) < \tau(x) + 1 = \tau(y) = \tau(z)$, then  say ${\triangle}$ is of type 3.
\end{enumerate}
The type of a node of $T$ is the same as the type of its corresponding triangle.
The root of $T$ corresponds to the triangle $\nu_1\nu_2\nu_3$
and the following are easy to observe.
\begin{enumerate}
\item[(a)] The root is of type 1.
\item[(b)] A node of type 1 has three children of type 2.
\item[(c)] A node of type 2 has one child of type 2 and two children of type 3.
\item[(d)] A node of type 3 has two children of type 3 and one child of type 1.
\end{enumerate}

For a triangle $\triangle$, define $\tau(\triangle)$ to be the minimum of $\tau(u)$  over  all $u\in V(\triangle)$.
Then it is easy to observe that, for two triangles $\overline{\triangle}$ and $\triangle$
of type 1 such that $\node{\overline{\triangle}}$ is an ancestor of $\node{\triangle}$
and there is no node of type 1 in the unique path connecting them,
we have $\tau({\triangle}) = \tau(\overline\triangle) + 1$.
This determines $\tau$ inductively:  for every $\node{\triangle} \in V(T)$, $\tau(\triangle)$ is one less than
the number of  nodes of type 1 in the path from $\node{\triangle}$ to the root.
We call  $\tau(\triangle)$ the \emph{auxiliary depth} of node $\triangle$,
and define the \emph{auxiliary height} of a tree $T$, written $\ah(T)$,
to be the maximum auxiliary depth of its nodes.
Note that the auxiliary height is always less than or equal to the height.
Also, for a vertex $v \in V(G)$,
if $\triangle$ is the triangle that $v$ subdivides,
then $\tau(v) = \tau(\triangle)+1$.
We augment the tree $T$ by adding specification of the type of
each node,
and we abuse notation and call the augmented tree the $\triangle$-tree of the RAN.
Hence,  the radius of the RAN is either $\ah(T)$ or $\ah(T)+1$.

Notice that instead of building $T$ from the RAN $G$,
one can think of the random $T$ as being generated in the following manner:
let $n\ge 3$ be a positive integer. Start with a single node as the root of $T$.
So long as the number of nodes is less than $3n-8$,
choose a leaf $v$ independently of previous choices and uniformly at random,
and add three leaves as children of $v$.
Once the number of nodes becomes $3n-8$,
add the information about the types using rules (a)--(d), as follows.
Let the root have type 1, and determine the types of other nodes in a top-down manner.
For a node of type 1, let its children have type 2.
For a node of type 2, select one of the children independently and uniformly at random,
let that child have type 2, and let the other two children have type 3.
Similarly, for a node of type 3, select one of the children independently of previous choices and uniformly at random,
let that child have type 1, and let the other two children have type 3.
Henceforth, we will forget about $G$ and focus on finding the auxiliary height of a random tree $T$ generated in this manner.

A major difficulty in analyzing the auxiliary height of the tree generated in the aforementioned manner
is that the branches of a node are heavily dependent, as the total number of nodes equals $3n-8$.
To remedy this we consider another process which has the desired independence and approximates the
original process well enough for our purposes.
The process, $\widehat{P}$,
starts with a single node, the root, which is born at time 0, and is of type 1.
From this moment onwards, whenever a node is born (say at time
$\kappa$), it waits for a random time $X$, which is distributed
exponentially with mean 1, and after time $X$ has passed (namely, at
absolute time $\kappa + X$) gives birth to three children, whose types
are determined as before (according to the rules  (b)--(d), and using
randomness whenever there is a choice) and dies. Moreover, the
lifetime of the nodes are independent.
 By the memorylessness of the exponential distribution, if
one starts looking at the process at any (deterministic) moment, the next leaf to die is
chosen uniformly at random.
For a nonnegative (possibly random) $t$,
we denote by ${\widehat{T}}^t$ the random almost surely finite tree obtained by taking a snapshot of this process at time $t$.
Hence, for any deterministic $t \geq 0$,
the distribution of ${\widehat{T}}^t$ conditional on
${\widehat{T}}^t$ having exactly $3n-8$ nodes, is the same as the
distribution of $T$.

\begin{lemma}
\label{lem:equal_logs}
Assume that there exists a constant $c$ such that
a.a.s.\ the auxiliary height of $\widehat{T}^t$ is asymptotic to $c t$ as $t\to\infty$.
Then the radius of a RAN with $n$ vertices is a.a.s.\ asymptotic to $c \log n / 2$ as $n\to\infty$.
\end{lemma}

\begin{proof}
Let $\ell_n = 3n - 8$, and let $\varepsilon>0$ be fixed.
For the process $\widehat{P}$, we define three stopping times as follows:
\begin{description}
\item $a_1$ is the deterministic time $(1-\varepsilon) \log (\ell_n) / 2$.
\item $A_2$ is the random time when the evolving tree has exactly $\ell_n$ nodes.
\item $a_3$ is the deterministic time $(1+\varepsilon) \log (\ell_n) / 2$.
\end{description}

Broutin and Devroye~\cite[Proposition~2]{treeheight} proved that
almost surely
$$\log |V(\widehat{T}^t)| \sim 2t\:,$$
which implies the same statement a.a.s.\ as $t\to\infty$.
This means that, as $n\to\infty$, a.a.s.
$$\log |V(\widehat{T}^{a_1})| \sim 2a_1 = (1 - \varepsilon) \log (\ell_n) \:,$$
and hence  $|V(\widehat{T}^{a_1})| < \ell_n$, which implies $a_1 < A_2$.
Symmetrically, it can be proved that a.a.s.\ as $n\to\infty$ we have $A_2 < a_3$.
It follows that a.a.s.\ as $n\to\infty$
$$\ah\left(\widehat{T}^{a_1}\right) \leq \ah \left(\widehat{T}^{A_2}\right) \leq \ah \left(\widehat{T}^{a_3}\right) \:.$$

By the assumption, a.a.s.\ as $n\to\infty$ we have $\ah\left(\widehat{T}^{a_1}\right) \sim (1-\varepsilon) c \log (\ell_n) / 2$
and
$\ah\left(\widehat{T}^{a_3}\right) \sim (1+\varepsilon) c \log (\ell_n) / 2$.
On the other hand, as noted above, $T$ has the same distribution as $\widehat{T}^{A_2}$.
It follows that a.a.s.\ as $n\to\infty$
$$ 1 - 2\varepsilon \leq \frac{2\ah(T)}{c \log (\ell_n)} \leq 1 + 2\varepsilon \:.$$
Since $\varepsilon$ was arbitrary, the result follows.
\end{proof}

It will be more convenient to view the process $\widehat{P}$ in the following equivalent way.
Denote by $\operatorname{Exp}(1)$ an exponential random variable with mean 1.
Let $\widehat{T}$ denote an infinite ternary tree whose nodes have types assigned using rules  (a)--(d) and
are associated with independent $\operatorname{Exp}(1)$ random variables.
For convenience, each edge of the tree from a parent to a child is labelled with the  random variable associated with the parent, which denotes the age of the parent when the child is born.
For every node $u \in V(\widehat{T})$, its \emph{birth time} is defined as the
sum of the labels on the edges connecting $u$ to the root,
and the birth time of the root is defined to be zero.
Given $t\geq 0$, the  tree $\widehat{T}^t$ is the  subtree induced by nodes
 with birth time  less than or equal to $t$,
and is finite with probability one.

Let $k\ge 3$ be a fixed positive integer.
We define two {random} infinite trees $\underline{T_k}$ and $\overline{T_k}$ as follows.
First, we regard $\widehat{T}$ as a tree generated by each node giving birth to exactly three children with types assigned using  (b)--(d),
and with an $\operatorname{Exp}(1)$ random variable used to label the edges to its children.
The tree $\underline{T_k}$ is obtained using the same generation rules as $\widehat{T}$ except that every  node   of type 2 or 3, whose
distance to its closest ancestor of type 1 is equal to $k$,
dies without giving birth to any children.
Given $t\geq 0$, the  random  (almost surely finite) tree $\underline{T^t_k}$ is, as before, the subtree of $\underline{T_k}$ induced
by nodes  with birth time less than or equal to $t$.
The tree $\overline{T_k}$ is also generated similarly to $\widehat{T}$,
except that for each node $u$ of type 2 (respectively, 3) in $\overline{T_k}$ whose
distance to its closest ancestor of type 1 equals $k$,
$u$ has exactly three (respectively, four) children of type 1,
and the edges joining $u$ to its children get label 0
instead of random ${\operatorname{Exp}}(1)$ labels.
(In the ``evolving tree'' interpretation,
$u$ immediately gives birth to three or four children of type 1 and dies.)
Such a node $u$ is called an \emph{annoying} node.
The  random (almost surely finite) tree $\overline{T^t_k}$ is defined as before.

\begin{lemma}
\label{lem:sandwich}
For every  fixed $k\ge 3$, every $t\geq 0$, and every $g=g(t)$, we have
$$ \p{\ah \left ( \underline{T^t_k}\right) \geq g} \leq \p{ \ah \left ( \widehat{T}^t \right ) \geq g}
\leq \p{ \ah \left ( \overline{T_k^t} \right) \geq g} \:.$$
\end{lemma}
\begin{proof}
The left inequality follows from the fact that the random edge labels of  $\widehat{T}$ and $\underline{T_{k}}$
can easily be coupled using a common sequence of independent $\operatorname{Exp}(1)$ random variables
in such a way that for every $t \geq 0$,
the generated $\underline{T_{k}^t}$ is always a subtree of the generated $\widehat{T}^t$.

For the right inequality, we use a sneaky coupling between the edge labels of  $\widehat{T}$ and $\overline{T_{k}}$.
It is enough to choose them using a common sequence of independent $\operatorname{Exp}(1)$ random variables $X_1,X_2,\ldots$
and define a one-to-one mapping $f : V(\widehat{T}) \rightarrow V\left(\overline{T_{k}}\right)$
such that for every $u \in V(\widehat{T})$,
\begin{enumerate}
\item[(1)] the auxiliary depth of $f(u)$ is greater than or equal to the auxiliary depth of $u$, and
\item[(2)] for some $I$ and $J \subseteq I$, the birth time of $u$ equals $\sum_{i\in I} X_i$
and the birth time of $f(u)$ equals $\sum_{j\in J} X_j$.
\end{enumerate}

For annoying nodes, the coupling and the mapping $f$ is shown down to their grandchildren in Figures~\ref{fig:type-2} and~\ref{fig:type-3}.
This is easily extended in a natural way to all other nodes of the tree.
\end{proof}

\begin{figure}
  \begin{center}
    \begin{tikzpicture}[level distance=2cm,level 1/.style={sibling distance = 3.75cm},level 2/.style={sibling distance = 1.3cm}]
\tikzset{
  treenode/.style = {circle,draw},
}
\tikzstyle{every label}=[draw=none,label distance=-3pt]
\def \la {-60};
\node [treenode](root){$2$}
child {node (a1)[treenode] {$2$}
  child {node (b1)[treenode,label={\la:$1$}] {$2$} edge from parent node [left]{$B$}
  }
  child {node (b2)[treenode,label={\la:$2$}] {$3$} edge from parent node [left]{$B$}
  }
  child {node (b3)[treenode,label={\la:$3$}] {$3$} edge from parent node [left]{$B$}
  }
  edge from parent node [above left] {$A$}
}
child {node (a2)[treenode] {$3$}
  child {node (c1)[treenode, label={\la:$4$}]{$1$}  edge from parent node [left] {$C$}
  }
  child {node (c2)[treenode,label={\la:$5$}] {$3$}  edge from parent node [left] {$C$}
  }
  child {node (c3)[treenode,label={\la:$6$}] {$3$}  edge from parent node [left] {$C$}
  }
  edge from parent node [left] {$A$}
}
child {node (a3)[treenode] {$3$}
  child {node (d1)[treenode,label={\la:$7$}]{$1$}  edge from parent node [left] {$D$}
  }
  child {node (d2)[treenode,label={\la:$8$}] {$3$}  edge from parent node [left] {$D$}
  }
  child {node (d3)[treenode,label={\la:$9$}] {$3$}  edge from parent node [left] {$D$}
  }
  edge from parent node [above right] {$A$}
};
\end{tikzpicture}

\vspace{2cm}

\begin{tikzpicture}[level distance=2cm,level 1/.style={sibling distance = 3.75cm},
level 2/.style={sibling distance = 3.75cm},
level 3/.style={sibling distance = 1.3cm}]
\tikzset{
  treenode/.style = {circle,draw},
}
\tikzstyle{every label}=[draw=none,label distance=-3pt]
\def \la {-60};

\node (root)[treenode]{$2$}
  child {node (a1)[treenode,label={\la:$4$}] {$1$}  edge from parent node [above left]{$0$}
  }
  child {node (a2) [treenode]{$1$}
    child {node (b1) [treenode]{$2$}
      child {node (c1)[treenode, label={\la:$1$}]{$2$}
        edge from parent node [left]{$B$}
      }
      child {node (c2)[treenode, label={\la:$2$}] {$3$}
       edge from parent node [left]{$B$}
      }
      child {node (c3)[treenode, label={\la:$3$}] {$3$}
        edge from parent node [left]{$B$}
      }
      edge from parent node [above left]{$A$}
    }
    child {node (b2) [treenode]{$2$}
      child {node (d1) [treenode]{$2$}
        edge from parent node [left]{$C$}
      }
      child {node (d2)[treenode, label={\la:$5$}] {$3$}
        edge from parent node [left]{$C$}
      }
      child {node (d3)[treenode, label={\la:$6$}] {$3$}
        edge from parent node [left]{$C$}
      }
      edge from parent node [left]{$A$}
    }
    child {node (b3) [treenode]{$2$}
      child {node (d1) [treenode]{$2$}
        edge from parent node [left]{$D$}
      }
      child {node (d2)[treenode, label={\la:$8$}] {$3$}
        edge from parent node [left]{$D$}
      }
      child {node (d3)[treenode, label={\la:$9$}] {$3$}
        edge from parent node [left]{$D$}
      }
      edge from parent node [above right]{$A$}
    }
    edge from parent node [left]{$0$}
  }
  child {node (a3)[treenode, label={\la:$7$}] {$1$}
    edge from parent node [above right]{$0$}
  };
  \end{tikzpicture}
\end{center}
\caption
{\label{fig:type-2}
Illustrating the coupling in Lemma~\ref{lem:sandwich} for an annoying node of type 2 in $\widehat{T}$.
The offspring of the node is shown above and the offspring of the corresponding node in $\overline{T_k}$ is shown below.
The type of each node is written inside the node.
The coupling of edge labels is defined by the appearance of $A,B,\dots$ in the two cases.
The label 0 is also used in the case of $\overline{T_k}$. The function $f$ is defined by the labels beside the nodes.}
\end{figure}

\begin{figure}
\begin{center}
\begin{tikzpicture}[level distance=2cm,level 1/.style={sibling distance = 3.75cm},
level 2/.style={sibling distance = 1.25cm}]
\tikzset{
  treenode/.style = {circle,draw},
}
\tikzstyle{every label}=[draw=none,label distance=-3pt]
\def \la {-60};
\node (root)[treenode]{$3$}
child {node (a1) [treenode]{$1$}
  child {node (b1)[treenode, label={\la:$1$}]{$2$}
    edge from parent node [left]{$F$}
  }
  child {node (b2)[treenode, label={\la:$2$}] {$2$}
    edge from parent node [left]{$F$}
  }
  child {node (b3)[treenode, label={\la:$3$}] {$2$}
    edge from parent node [left]{$F$}
  }
  edge from parent node [above left]{$E$}
}
child {node (a2) [treenode]{$3$}
  child {node (c1)[treenode,label={\la:$4$}]{$1$}
    edge from parent node [left]{$G$}
  }
  child {node (c2)[treenode,label={\la:$5$}] {$3$}
    edge from parent node [left]{$G$}
  }
  child {node (c3)[treenode,label={\la:$6$}] {$3$}
    edge from parent node [left]{$G$}
  }
  edge from parent node [left]{$E$}
}
child {node (a3)[treenode] {$3$}
  child {node (d1)[treenode, label={\la:$7$}]{$1$}
    edge from parent node [left]{$H$}
  }
  child {node (d2)[treenode, label={\la:$8$}] {$3$}
    edge from parent node [left]{$H$}
  }
  child {node (d3)[treenode, label={\la:$9$}] {$3$}
    edge from parent node [left]{$H$}
  }
  edge from parent node [above right]{$E$}
};
\end{tikzpicture}

\vspace{2cm}

\begin{tikzpicture}[level distance=2cm,level 1/.style={sibling distance = 3.75cm},
    level 3/.style={sibling distance = 1.25cm}]]
    \tikzset{
      treenode/.style = {circle,draw},
    }
    \tikzstyle{every label}=[draw=none,label distance=-3pt]
    \def \la {-60};
    \node (root)[treenode]{$3$}
    child {node (a1)[treenode]{$1$}
      [sibling distance=1.25cm]
      child {node (b1)[treenode, label={\la:$1$}]{$2$}
        edge from parent node [left]{$F$}
      }
      child {node (b2)[treenode, label={\la:$2$}] {$2$}
        edge from parent node [left]{$F$}
      }
      child {node (b3)[treenode, label={\la:$3$}] {$2$}
        edge from parent node [left]{$F$}
      }
      edge from parent node [above left]{$0$}
    }
    child {node (a2)[treenode, label={\la:$4$}]{$1$}
      edge from parent node [left]{$0$}
    }
    child {node (a3)[treenode]{$1$}
      [sibling distance=3.75cm]
      child {node (c1)[treenode]{$2$}
        child {node (d1)[treenode]{$2$}
          edge from parent node [left]{$G$}
        }
        child {node (d2)[treenode, label={\la:$5$}] {$3$}
          edge from parent node [left]{$G$}
        }
        child {node (d3)[treenode, label={\la:$6$}]{$3$}
          edge from parent node [left]{$G$}
        }
        edge from parent node [above left]{$E$}
      }
      child {node (c2) [treenode]{$2$}
        child {node (e1)[treenode]{$2$}
            edge from parent node [left]{$H$}
          }
          child {node (e2)[treenode, label={\la:$8$}] {$3$}
            edge from parent node [left]{$H$}
          }
          child {node (e3)[treenode, label={\la:$9$}]{$3$}
            edge from parent node [left]{$H$}
          }
          edge from parent node [left]{$E$}
        }
        child {node (c3) [treenode]{$2$}
          child {node (f1)[treenode]{$2$}
            edge from parent node [left]{$I$}
          }
          child {node (f2)[treenode]{$3$}
            edge from parent node [left]{$I$}
          }
          child {node (f3)[treenode]{$3$}
            edge from parent node [left]{$I$}
          }
          edge from parent node [above right]{$E$}
        }
        edge from parent node [left]{$0$}
      }
      child {node (a4)[treenode, label={\la:$7$}] {$1$}
        edge from parent node [above right]{$0$}};
    \end{tikzpicture}
\end{center}
\caption{ \label{fig:type-3}
The coupling in Lemma~\ref{lem:sandwich} for an annoying node of type 3 in $\widehat{T}$.
}
\end{figure}

 With a view to proving Lemma~\ref{lem:radius} by appealing to Lemmas~\ref{lem:equal_logs}~and~\ref{lem:sandwich},
we will define two sequences
$\left(\underline{\rho_k}\right)$
and
$\left(\overline{\rho_k}\right)$
such that for each $k$, a.a.s.\ the heights of $\overline{T_k^t}$ and $\underline{T_k^t}$
are asymptotic to $\overline{\rho_k} t $ and $\underline{\rho_k}t$, respectively,
and {also}
$$
\lim_{k\to\infty} \underline{\rho_k} = \lim_{k\to\infty} \overline{\rho_k} = c \:,
$$
where $c\approx 1.668$ is defined in the statement of Lemma~\ref{lem:radius}.

For the rest of this section, asymptotics are with respect to $t$ instead of $n$, unless otherwise specified.
We analyze the heights of $\underline{T_k}$ and $\overline{T_k}$
with the help of a theorem of Broutin and Devroye~\cite[Theorem~1]{treeheight}.
We state here a special case suitable for our purposes, including a trivial correction to the conditions on $E$.
\begin{theorem}
\label{thm:broutin-devroye}
Let $E$ be a prototype nonnegative random variable that
satisfies $\p{E=0} = 0$ and $\sup \{z : \p{E > z} =1 \} = 0$,
and such that $\p{E=z}<1$ for every $z\in\mathbb{R}$;
and for which there exists $\lambda>0$ such that $\e{\exp(\lambda E)}$ is finite.
Let $b$ be a fixed positive integer greater than {1} and let $T_{\infty}$ be an infinite $b$-ary tree.
Let $B$ be a prototype random $b$-vector with each component distributed as $E$  (but not necessarily independent components). For every node $u$ of $T_{\infty}$, label the edges to the children of $u$ using an independently generated copy of $B$.

Given $t\geq 0$,
let $H_t$ be the height of the subtree of $T_{\infty}$ induced by
the nodes for which the sum of the labels on their path to the root is at most $t$.
Then, a.a.s.\ we have
${H_t} \sim \rho t $,
where $\rho$ is the unique solution to
$$ \sup \{\lambda / \rho - \log (\e{\exp(\lambda E)}) : \lambda \leq 0 \} = \log b \:. $$
\end{theorem}

For each $i=2,3,\dots$, let $\alpha_i, \beta_i, \gamma_i$ denote the
number of nodes of type 1, 2, 3 at depth $i$ of $\widehat{T}$
for which the root is the only {node of type 1} in their path to the root.
Then rules  (a)--(d)  for determining node types imply
$$\forall i>2 \qquad \alpha_i = \gamma_{i-1}, \quad \beta_i = \beta_{i-1}, \quad \gamma_i = 2\beta_{i-1}+2\gamma_{i-1} \:.$$
These, together with $\alpha_2 = 0$, $\beta_2 = 3$, and $\gamma_2 = 6$, imply
\begin{equation}
\forall i \geq 2 \qquad \alpha_i = 3 \times 2^{i-1} - 6, \quad \beta_i = 3, \quad \gamma_i = 3 \times 2^i - 6 \:.\label{alpha's}
\end{equation}
Let $\underline{b_k}=\sum_{i=1}^k \alpha_i$
and $\overline{b_k}=\sum_{i=1}^k \alpha_i + 3\beta_k+4\gamma_k$.
For a positive integer $s$,
let $\operatorname{Gamma}(s)$ denote the Gamma distribution with mean $s$,
i.e., the distribution of the sum of $s$ independent $\operatorname{Exp}(1)$ random variables.

We define a random infinite tree $\underline{T_k}'$ as follows.
The nodes of $\underline{T_k}'$ are the type-1 nodes of $\underline{T_k}$.
Let $V'$ denote the set of these nodes.
For $u,v\in V'$ such that $u$ is the closest type-1 ancestor of $v$ in $\underline{T_k}$,
 there is an edge joining $u$ and $v$ in $\underline{T_k}'$, whose label
equals the sum of the labels of the edges in the unique $(u,v)$-path in $\underline{T_k}$.
By the construction, for all $t\geq 0$, the height of
the subtree of $\underline{T_k}'$ induced by nodes
 with birth time  less than or equal to $t$ equals the auxiliary height of $\underline{T^t_k}$.
Let $u$ be a node in $\underline{T_k}'$.
Then observe that  for each $i = 3, 4, \dots, k$,
$u$ has $\alpha_i$ children whose birth times equal the birth time of $u$
plus a $\operatorname{Gamma}(i)$ random variable.
In particular, $\underline{T_k}'$ is an infinite $\underline{b_k}$-ary tree.

To apply Theorem~\ref{thm:broutin-devroye} we need the label of each edge to have the same distribution.
For this, we  create a random rearrangement of  $\underline{T_k}'$. First
let $\underline{E_k}$ be the random variable such that for each $3\le i\le k$,
with probability $\alpha_i/\underline{b_k}$,
$\underline{E_k}$ is distributed as a $\operatorname{Gamma}(i)$ random variable.
Now, for each node $u$ of $\underline{T_k}'$, starting from the root and in a top-down manner,
randomly permute the branches below $u$.
This results in an infinite $\underline{b_k}$-ary tree, every edge of which has a random label distributed as  $\underline{E_k}$.
Although    the labels of edges from a node to its children are dependent,
the $\underline{b_k}$-vector of labels of edges from a node to its children is independent of all other edge labels,
as required for Theorem~\ref{thm:broutin-devroye}.
Let $\rho$ be the solution to
\begin{equation}
\label{eq:rhounder}
\sup \{\lambda / \rho - \log (\e{\exp(\lambda \underline{E_k})}) : \lambda \leq 0 \} = \log \underline{b_k} \:.
\end{equation}
Then by Theorem~\ref{thm:broutin-devroye}, a.a.s.\ the auxiliary height of $\underline{T^t_k}$,
which equals the height of the subtree of $\underline{T_k}'$ induced by nodes
 with birth time less than or equal to $t$, is asymptotic to $\rho t$.
Notice that we have
$$ \e{\exp(\lambda \: \operatorname{Exp}(1))} = \frac{1}{1-\lambda} \:.$$
So, by the definition of $\operatorname{Gamma}(s)$, and since the product of expectation
of independent variables equals the expectation of their product,
$$ \e{\exp(\lambda \: \operatorname{Gamma}(s))} = \frac{1}{(1-\lambda)^s} \:.$$
Hence by linearity of expectation,
\begin{equation}
\label{eq:genunder}
\e{\exp(\lambda \underline{E_k})}=
\sum_{i=3}^k  \frac{ \alpha_i}{\underline{b_k} (1-\lambda)^{i} }\:.
\end{equation}

One can define a  random infinite $\overline{b_k}$-ary tree $\overline{T_k}'$ in a similar way.
Let $\overline{E_k}$ be the random variable such that for each $3\le i\le k-1$,
with probability $\alpha_i/\overline{b_k}$,
it is distributed as a $\operatorname{Gamma}(i)$ random variable,
and with probability $(\alpha_k + 3\beta_k + 4\gamma_k) /  \overline{b_k}$,
it is distributed as a $\operatorname{Gamma}(k)$ random variable.
Then by a similar argument,
a.a.s.\ the auxiliary height of $\overline{T^t_k}$ is asymptotic to $\rho t$,
where $\rho$ is the solution to
\begin{equation}
\label{eq:rhoover}
\sup \{\lambda / \rho - \log (\e{\exp(\lambda \overline{E_k})}) : \lambda \leq 0 \} = \log \overline{b_k} \:.
\end{equation}
Moreover, one calculates
\begin{equation}
\label{eq:genover}
\e{\exp(\lambda \overline{E_k})}=
\frac{ \alpha_k + 3\beta_k + 4\gamma_k}{\overline{b_k}(1-\lambda)^{k}}
+\sum_{i=3}^{k-1} \frac{ \alpha_i}{\overline{b_k} (1-\lambda)^{ i} }\:.
\end{equation}

As part of our plan to prove Lemma~\ref{lem:radius},
we would like to define
$\underline{\rho_k}$ and $\overline{\rho_k}$
in such a way that
they are the unique solutions to~(\ref{eq:rhounder})~and~(\ref{eq:rhoover}), respectively.
We first need to establish two analytical lemmas.

For later convenience, we define $\mathcal F$ to be the set of positive functions $f : [0.1, 0.2] \rightarrow \mathbb{R}$
that are differentiable on $(0.1,0.2)$,
and let $W : \mathcal{F} \rightarrow \mathbb{R}^{[0.1,0.2]}$ be the operator defined as
$$Wf (x) = x (x-1) f'(x) / f(x) - \log f(x) \:.$$
Note that $Wf$ is continuous.
Define $h \in \mathcal{F}$ as
$$h(x) = \frac{12x^3}{1-2x} - \frac{6x^3}{1-x} \:.$$

\begin{lemma}
The function $Wh$ has a unique root $\hat x$ in $(0.1,0.2)$.
\end{lemma}

\begin{proof}
By the definition of $(\alpha_i)_{i\ge 3}$ in (\ref{alpha's}) we have
$$
h(x)=\sum_{i\ge 3}\alpha_i x^i \qquad \forall x \in [0.1,0.2] \:.
$$
Since $\alpha_i>0$ for all $i\ge 3$, we have $h(x)>0$ and $h'(x)>0$ for $x \in [0.1,0.2]$, and hence  the derivative of $\log h(x)$ is positive.
Moreover, the derivative of $x(x-1)h'(x)/h(x)$ equals $4x(x-1)/(1-2x)^2$, which is negative.
Therefore, $Wh(x)$ is a strictly decreasing function on $[0.1,0.2]$.
Numerical calculations give
$Wh (0.1) \approx 1.762 > 0$
and
$Wh(0.2) \approx -0.831 < 0$.
Hence, there is a unique solution to $Wh(x)=0$ in $(0.1,0.2)$.
\end{proof}

\begin{remark}
Numerical calculations give $\hat x \approx  0.1629562 \:.$
\end{remark}

Define functions $\underline{g_k}, \overline{g_k} \in \mathcal{F}$ as
$$ \underline{g_k}(x)=\sum_{i=3}^k \alpha_i x^i, \mathrm{\ and\ }\overline{g_k}(x)=(\alpha_k+3\beta_k+4\gamma_k)x^k + \sum_{i=3}^{k-1} \alpha_i x^i\:.$$
Note that by (\ref{eq:genunder}) and (\ref{eq:genover}),
\begin{equation}
\label{eq:gs}
\underline{b_k} \: \e{\exp\left(\lambda \underline{E_k}\right)} = \underline{g_k}\left(\frac{1}{1-\lambda}\right)
, \mathrm{\ and\ }
\overline{b_k} \: \e{\exp\left(\lambda \overline{E_k}\right)} = \overline{g_k}\left(\frac{1}{1-\lambda}\right)
\end{equation}
hold at least when $(1-\lambda)^{-1} \in [0.1,0.2]$,
namely for all $\lambda \in [-9, -4]$.

\begin{lemma}
\label{lem:limits}
Both sequences $\left(W \underline{g_k}\right)_{k=3}^{\infty}$
and $\left(W \overline{g_k}\right)_{k=3}^{\infty}$
converge pointwise to $Wh$ on $[0.1,0.2]$ as $k\rightarrow \infty$.
Also, there exists a positive integer $k_0$ and
sequences $\left(\underline{x_k}\right)_{k=k_0}^{\infty}$ and $\left(\overline{x_k}\right)_{k=k_0}^{\infty}$
such that $W\underline{g_k}\left(\underline{x_k}\right) = W\overline{g_k}\left(\overline{x_k}\right) = 0$
for all $k \geq k_0$, and
$$\lim_{k\rightarrow \infty} \underline{x_k} = \lim_{k\rightarrow \infty} \overline{x_k} = \hat x \:.$$
\end{lemma}

\begin{proof}
For any $x \in [0.1,0.2]$, we have
$$\lim_{k\rightarrow \infty} \underline{g_k}(x) = h(x), \quad \lim_{k\rightarrow \infty} \underline{g_k}'(x) = h'(x),  \quad\lim_{k\rightarrow \infty} \overline{g_k}(x) = h(x), \quad \lim_{k\rightarrow \infty} \overline{g_k}'(x) = h'(x) \:,$$
so the sequences
$\left(W \underline{g_k}\right)_{k=3}^{\infty}$
and
$\left(W \overline{g_k}\right)_{k=3}^{\infty}$
converge pointwise to $Wh$.

Next, we show the existence of a positive integer $\underline{k_0}$
and a sequence $\left(\underline{x_k}\right)_{k=k_0}^{\infty}$
such that $W\underline{g_k}\left(\underline{x_k}\right) = 0$
for all $k \geq k_0$, and
$$\lim_{k\rightarrow \infty} \underline{x_k} = \hat x \:.$$
The proof for existence of corresponding positive integer $\overline{k_0}$
and the sequence $\left(\overline{x_k}\right)_{k=k_0}^{\infty}$
is similar, and we may let $k_0 = \max \{\underline{k_0}, \overline{k_0}\}$.

Since $Wh (0.1) > 0$ and $Wh (0.2) < 0$, there exists $\underline{k_0}$ so that for $k \geq \underline{k_0}$,
$W\underline{g_k} (0.1) > 0$ and $W\underline{g_k} (0.2) < 0$.
Since $W\underline{g_k}$ is continuous for all $k\geq 3$,
it has at least one root in $(0.1,0.2)$.
Moreover, since $W\underline{g_k}$ is continuous, the set
$\{x : W\underline{g_k} (x) = 0\}$ is a closed set,
thus we can choose a  root $\underline{x_k}$ closest to $\hat x$.
We just need to show that $\lim_{k\rightarrow \infty} \underline{x_k} = \hat x$.
Fix an $\varepsilon>0$.
Since $Wh\left(\hat x - \varepsilon\right) >0 $ and $Wh\left(\hat x + \varepsilon\right)<0$,
there exists a large enough $M$ such that for all $k\geq M$,
$W\underline{g_k}(\hat x - \varepsilon) >0 $ and $W\underline{g_k}(\hat x + \varepsilon)<0$. Thus $\underline{x_k} \in (\hat x - \varepsilon, \hat x + \varepsilon)$.
Since $\varepsilon$ was arbitrary, we conclude that
$\lim_{k\rightarrow \infty} \underline{x_k} = \hat x$.
\end{proof}

Let $k_0$ be as in Lemma~\ref{lem:limits} and
let
$\left(\underline{x_k}\right)_{k=k_0}^{\infty}$ and $\left(\overline{x_k}\right)_{k=k_0}^{\infty}$
be the sequences given by Lemma~\ref{lem:limits}.
Define the sequences
$\left(\underline{\rho_k}\right)_{k=k_0}^{\infty}$ and $\left(\overline{\rho_k}\right)_{k=k_0}^{\infty}$
by
\begin{equation}\label{eq:rho}
\underline{\rho_k}=\left(1-\underline{x_k}^{-1}\right)/\log \underline{g_k}(\underline{x_k}),\quad
\overline{\rho_k}=\left(1-{\overline{x_k}\:}^{-1}\right)/\log \overline{g_k}(\overline{x_k}).
\end{equation}

\begin{lemma}
\label{lem:long}
For every fixed $k\geq k_0$,
a.a.s.\ the heights of $\overline{T_k^t}$ and $\underline{T_k^t}$
are asymptotic to $\overline{\rho_k} t $ and $\underline{\rho_k}t$, respectively.
\end{lemma}
\begin{proof}

We give the argument for $\overline{T_k^t}$;
the argument for $\underline{T_k^t}$ is similar.
First of all, we claim that
$\log (\e{\exp(\lambda \overline{E_k})})$
is a strictly convex function of $\lambda$ over $(-\infty,0]$.
To see this, let $\lambda_1 < \lambda_2 \leq 0$
and let $\theta \in (0,1)$.
Then we have
\begin{align*}
\e{\exp\left(\theta \lambda_1 \overline{E_k} + (1-\theta) \lambda_2 \overline{E_k} \right)}
& = \e{\left[\exp\left(\lambda_1 \overline{E_k}\right)\right]^{\theta}
\left[\exp\left( \lambda_2 \overline{E_k} \right)\right]^{1-\theta}} \\
& <
\e{\exp\left(\lambda_1 \overline{E_k}\right)}^{\theta}
\e{\exp\left(\lambda_2 \overline{E_k}\right)}^{1-\theta} \:,
\end{align*}
where the inequality follows from H\"{o}lder's inequality,
and is strict as the random variable $\overline{E_k}$ does not have all of its mass concentrated in a single point.
Taking logarithms completes the proof of the claim.

It follows that given any value of $\rho$, $\lambda / \rho - \log (\e{\exp(\lambda \overline{E_k})})$
is a strictly concave function of $\lambda \in (-\infty,0]$
and hence attains it supremum at a unique $\lambda \le 0$.

Now, define
$$\overline{\lambda_k} = 1 - {\overline{x_k}\:}^{-1} \:,$$
which is in $(-9,-4)$ as $\overline{x_k} \in (0.1, 0.2)$.
Next we will show that
\begin{align}
\overline{\lambda_k} / \overline{\rho_k} - \log (\e{\exp(\overline{\lambda_k} \,\overline{E_k})}) & = \log \overline{b_k}
\label{eq:rho1}
\:, \\
\frac{\mathrm{d}}{\mathrm{d}{\lambda}}
\left[\lambda /  \overline{\rho_k} - \log (\e{\exp(\lambda \overline{E_k})})\right] \Big|_{\lambda =  \overline{\lambda_k}} & = 0
\label{eq:rho2}
\:,
\end{align}
which implies that $\overline{\rho_k}$ is the unique solution for
(\ref{eq:rhoover}),
and thus by Theorem~\ref{thm:broutin-devroye} and the discussion after it,
the height of $\overline{T_k^t}$ is asymptotic to $\overline{\rho_k} t$.

Notice that $\overline{\lambda_k} \in (-9,-4)$, so by (\ref{eq:gs}),
$$\overline{b_k} \e{\exp({\lambda} \overline{E_k})} = \overline{g_k}((1-\lambda)^{-1})$$
for $\lambda$ in a sufficiently small open neighbourhood of $\overline{\lambda_k}$.
Taking logarithm of   both sides and using~(\ref{eq:rho}) gives~(\ref{eq:rho1}).

To prove (\ref{eq:rho2}), note that
\begin{align*}
\frac{\mathrm{d}}{\mathrm{d}{\lambda}} \left[\log (\e{\exp(\lambda \overline{E_k})})\right] \Big|_{\lambda = \overline{\lambda_k}} & =
\frac{\mathrm{d}}{\mathrm{d}{\lambda}} \left[\log  \overline{g_k} \left(\left(1-\overline{\lambda_k}\right)^{-1}\right) - \log  \overline{b_k} \right]\Big|_{\lambda = \overline{\lambda_k}} \\
& = \frac {  {\overline{g_k}\,}' \left( (1-\overline{\lambda_k})^{-1} \right)} {\left(1-\overline{\lambda_k}\right)^{2} \: \overline{g_k} \left( (1-\overline{\lambda_k})^{-1} \right)} = {\overline{x_k}\:}^2 \: \frac{{\overline{g_k}\,}'(\overline{x_k})}{\overline{g_k}(\overline{x_k})} \:.
\end{align*}
By Lemma~\ref{lem:limits}, $W\overline{g_k} (\overline{x_k}) = 0$, i.e.,
$${\overline{x_k}\:}^2 \: \frac{{\overline{g_k}\,}'(\overline{x_k})}{{\overline{g_k}}(\overline{x_k})} =
{\overline{x_k}\:}^2 \: \frac{\log \overline{g_k}( \overline{x_k} )}{\overline{x_k} (\overline{x_k} - 1)}
= \frac{\log \overline{g_k}(\overline{x_k})}{1 - {\overline{x_k}\:}^{-1}} = \frac{1}{\overline{\rho_k}}\:,$$
and (\ref{eq:rho2}) is proved.
\end{proof}

We now have all the ingredients to prove Lemma~\ref{lem:radius}.

\begin{proof}
[Proof of Lemma~\ref{lem:radius}.]
By Lemma~\ref{lem:equal_logs}, we just need to show that
a.a.s.\ the auxiliary height of $\widehat{T}^t$ is asymptotic to $c t$, where
$$c = \frac{1-\hat {x}}{\log h(\hat x)} \:.$$
By Lemma~\ref{lem:long},
a.a.s.\ the heights of $\overline{T_k^t}$ and $\underline{T_k^t}$
are asymptotic to $\overline{\rho_k} t$ and $\underline{\rho_k} t$, respectively.
By Lemma~\ref{lem:limits}, $\overline{x_k}\to \hat x$ and $\underline{x_k}\to \hat x$.
Observe that $\left(\underline{g_k}\right)_{k=3}^{\infty}$ and $\left(\overline{g_k}\right)_{k=3}^{\infty}$
converge pointwise to $h$, and that
for every $k\geq 3$ and every $x\in[0.1,0.2]$,
$\underline{g_k}(x) \leq \underline{g_{k+1}}(x)$
and
$\overline{g_k}(x) \geq \overline{g_{k+1}}(x)$.
Thus by Dini's theorem (see, e.g., Rudin~\cite[Theorem~7.13]{rudin}),
$\left(\underline{g_k}\right)_{k=3}^{\infty}$ and $\left(\overline{g_k}\right)_{k=3}^{\infty}$
converge uniformly to $h$ on $[0.1,0.2]$.

Hence,
$$
\lim_{k\to\infty}\underline{\rho_k}=
\lim_{k\to\infty} \frac{1-\underline{x_k}^{-1}}{\log \underline{g_k}(\underline{x_k})} =
\frac{1-\hat {x}^{-1}}{\log h(\hat x)} = c \:,
$$
and
$$
\lim_{k\to\infty} \overline{\rho_k}=
\lim_{k\to\infty} \frac{1-{\overline{x_k}}^{-1}}{\log \overline{g_k}(\overline{x_k})} =
\frac{1-\hat {x} ^{-1}}{\log h(\hat x)} = c \:.
$$
It follows from Lemma~\ref{lem:sandwich} that
a.a.s.\ the auxiliary height of $\widehat{T}^t$ is asymptotic to $c t$, as required.
\end{proof}

\bibliographystyle{pagenduns}
\bibliography{rans}

\end{document}